\documentclass[reqno, a4paper]{amsart}

\usepackage{fixltx2e}
\usepackage[T1]{fontenc}
\usepackage[british]{babel}
\usepackage{mathabx}
\usepackage{hyperref}
\usepackage[all]{xy}

\theoremstyle{plain}
\newtheorem{theorem}{Theorem}
\newtheorem{lemma}{Lemma}

\newcommand{\som}[2]{\bigl\langle\begin{smallmatrix} {#1} \\ {#2}\end{smallmatrix}\bigr\rangle}
\newcommand{\bigsom}[2]{\left\langle\begin{smallmatrix} {#1} \\ {#2}\end{smallmatrix}\right\rangle}
\newcommand{\somm}[3]{\left\langle\begin{smallmatrix} {#1} \\ {#2} \\ {#3}\end{smallmatrix}\right\rangle}

\newcommand{\comp}{\raisebox{0.2mm}{\ensuremath{\scriptstyle{\circ}}}}
\newcommand{\defn}{\textbf}
\newcommand{\meet}{\wedge}
\newcommand{\join}{\vee}
\newcommand{\tensor}{\diamond}

\renewcommand{\Im}{\ensuremath{\mathrm{Im}}}
\renewcommand{\ker}{\ensuremath{\mathrm{ker}}}
\newcommand{\Ker}{\ensuremath{\mathrm{Ker}}}

\def\pullback{
 \ar@{-}[]+R+<6pt,-1pt>;[]+RD+<6pt,-6pt>%
 \ar@{-}[]+D+<1pt,-6pt>;[]+RD+<6pt,-6pt>}
 
\def\splitpullback{%
 \ar@{-}[]+R+<6pt,-.51ex>;[]+RD+<6pt,-6pt>%
 \ar@{-}[]+D+<.51ex,-6pt>;[]+RD+<6pt,-6pt>}

\newdir{>}{{}*:(1,-.2)@^{>}*:(1,+.2)@_{>}}
\newdir{<}{{}*:(1,+.2)@^{<}*:(1,-.2)@_{<}}
\newdir{ |>}{{}*!/-3.5pt/@{|}*!/-8pt/:(1,-.2)@^{>}*!/-8pt/:(1,+.2)@_{>}}
\newdir{ >>}{{}*!/8pt/@{|}*!/3.5pt/:(1,-.2)@^{>}*!/3.5pt/:(1,+.2)@_{>}}
\newdir{ >}{{}*!/-8pt/@{>}}

\begin{document}

\date{\today}

\author{Nelson Martins-Ferreira}
\author{Tim Van~der Linden}

\thanks{The first author was supported by IPLeiria/ESTG-CDRSP and Funda\c c\~ao para a Ci\^encia e a Tecnologia (grants SFRH/BPD/4321/2008 and PTDC/MAT/120222/2010)}
\thanks{The second author works as \emph{charg\'e de recherches} for Fonds de la Recherche Scientifique--FNRS. His research was supported by Centro de Matem\'atica da Universidade de Coimbra and by Funda\c c\~ao para a Ci\^encia e a Tecnologia (grant SFRH/BPD/38797/2007). He wishes to thank CMUC, IPLeiria and UCT for their kind hospitality during his stays in Coimbra, in Leiria and in Cape Town}

\email{martins.ferreira@ipleiria.pt}
\email{tim.vanderlinden@uclouvain.be}

\address{Departamento de Matem\'atica, Escola Superior de Tecnologia e Gest\~ao, Centro para o Desenvolvimento R\'apido e Sustentado do Produto, Instituto Poli\-t\'ecnico de Leiria, Leiria, Portugal}
\address{CMUC, Department of Mathematics, University of Coimbra, 3001--454 Coimbra, Portugal}
\address{Institut de recherche en math\'ematique et physique, Universit\'e catholique de Louvain, chemin du cyclotron~2 bte~L7.01.02, B-1348 Louvain-la-Neuve, Belgium}

\dedicatory{Dedicated to George Janelidze on the occasion of his sixtieth birthday}



\title[A decomposition formula for the weighted commutator]{A decomposition formula for\\ the weighted commutator}

\begin{abstract}
We decompose the weighted subobject commutator of M.~Gran, G.~Janelidze and A.~Ursini as a join of a binary and a ternary commutator.
\end{abstract}

\maketitle

In their article~\cite{GJU}, M.~Gran, G.~Janelidze and A.~Ursini introduce a \emph{weighted normal commutator} which, depending on the chosen weight, captures classical commutators such as the Huq commutator~\cite{Huq, Bourn-Huq, Borceux-Bourn} and the Smith commutator~\cite{Smith, Pedicchio, Bourn-Gran-Maltsev, Borceux-Bourn}. It is constructed as the normal closure of a so-called \emph{weighted subobject commutator}. We show how this latter commutator may be decomposed as a join of a binary and a ternary commutator~\cite{HVdL, Actions} defined in terms of co-smash products~\cite{Smash}. We moreover explain that the corresponding concept of \emph{weighted centrality} of arrows can be expressed in terms of the  \emph{admissibility} of certain diagrams in the first author's sense~\cite{MF-PhD}. 

\subsection*{The weighted subobject commutator}\label{The weighted subobject commutator}
In a finitely cocomplete homological category~\cite{Borceux-Bourn, Janelidze-Marki-Tholen}, a \defn{weighted cospan} is a triple of morphisms
\begin{equation}\label{weighted cospan}
\vcenter{\xymatrix@!0@=4em{& W \ar[d]^-{w} \\
X \ar[r]_-{x} & D & Y \ar[l]^-{y}}
}\end{equation}
in which $(x,y)$ plays the role of cospan and $w$ is the weight. Consider the pullback
\[
\vcenter{\xymatrix@!0@R=4em@C=4em{& W+Y \ar[rd]^-{\som{1_{W}}{0}}\\
(W+X)\times_W(W+Y) \ar[ru]^-{\pi_2} \ar[rd]_-{\pi_1} && W \\
& W+X \ar[ru]_-{\som{1_{W}}{0}}}}
\]
and the induced outer diagram
\[
\vcenter{\xymatrix@!0@R=4em@C=15em{W+X \ar[r]^-{\bigl\langle 1_{W+X},\iota_{W}\circ\som{1_{W}}{0}\bigr\rangle} \ar[rd]_-{\som{w}{x}} & (W+X)\times_W(W+Y)
\ar@{.>}[d]^-{\varphi} & W+Y \ar@<-.5ex>[l]_-{\bigl\langle\iota_{W}\circ\som{1_{W}}{0},1_{W+Y}\bigr\rangle} \ar[ld]^-{\som{w}{y}}\\
& D.}}
\]
In~\cite{GJU} the morphisms $x$ and $y$ are said to \defn{commute over $w$} if and only if there exists a dotted arrow $\varphi$ (called an \defn{internal multiplication}) such that the above diagram is commutative.

As explained in~\cite{GJU}, taking $W=0$ captures commuting pairs in the Huq sense ($x$ and $y$ commute over~$0$ if and only if they Huq-commute), and $w=1_{D}$ captures centralising equivalence relations in the Smith sense (the respective normalisations~$x$ and~$y$ of two equivalence relations~$R$ and~$S$ on~$D$ commute over~$1_{D}$ if and only if~$R$ and $S$ Smith-commute).

Now consider the canonical comparison morphism
\[
\left\langle\begin{smallmatrix} \iota_{W} & \iota_{W}\\
\iota_{X} & 0\\
0 & \iota_{Y}\end{smallmatrix}\right\rangle\colon W+X+Y \to (W+X)\times_W(W+Y)
\]
which, being a regular epimorphism~\cite{GJU} as the comparison between a sum and a product in the category of points over an object $W$ in a regular Mal'tsev category, induces a short exact sequence
\begin{equation}\label{Kernel K}
\xymatrix@=3em{0 \ar[r] & K \ar@{{ |>}->}[r] & W+X+Y \ar@{-{ >>}}[r]^-{\left\langle\begin{smallmatrix} \iota_{W} & \iota_{W}\\
\iota_{X} & 0\\
0 & \iota_{Y}\end{smallmatrix}\right\rangle} & (W+X)\times_W(W+Y) \ar[r] & 0.}
\end{equation}
The \defn{$(W,w)$-weighted subobject commutator $\kappa\colon{[(X,x),(Y,y)]_{(W,w)}\to D}$ of $x$ and $y$} is the direct image of $K$ along the induced arrow to $D$ as in
\[
\xymatrix{K \ar@{{ |>}->}[r] \ar@{.{ >>}}[d] & W+X+Y \ar@{->}[d]^-{\left\langle\begin{smallmatrix} w\\
x\\
y\end{smallmatrix}\right\rangle}\\
[(X,x),(Y,y)]_{(W,w)} \ar@{{ >}.>}[r]_-{\kappa} & D.}
\]
It is clear from the exactness of the above sequence that $x$ and $y$ commute over $w$ if and only if $[(X,x),(Y,y)]_{(W,w)}$ vanishes.

The normal closure of $\kappa$ is called the \defn{$(W,w)$-weighted normal commutator} of $x$ and $y$ and denoted by $N[(X,x),(Y,y)]_{(W,w)}$.

\subsection*{Admissibility}
In order to analyse the weighted subobject commutator in terms of the binary and ternary commutators considered in~\cite{HVdL, Actions}, we pass via an intermediate notion from~\cite{MF-PhD}. An \defn{admissibility diagram} is a diagram of shape
\begin{equation}\label{adm}
\vcenter{\xymatrix@!0@=4em{A \ar@<.5ex>[r]^-{f} \ar[rd]_-{\alpha} & B
\ar@<.5ex>[l]^-{r}
\ar@<-.5ex>[r]_-{s}
\ar[d]^-{\beta} & C \ar@<-.5ex>[l]_-{g} \ar[ld]^-{\gamma}\\
& D}}
\end{equation}
with $f\comp r=1_{B}=g\comp s$ and $\alpha\comp r=\beta=\gamma\comp s$. Note that by taking the pullback of $f$ with $g$, any admissibility diagram such as~\eqref{adm} may be extended to
\begin{equation*}\label{kite}
\vcenter{\xymatrix@!0@=3em{ & C \ar@<.5ex>[ld]^-{e_2} \ar@<-.5ex>[rd]_-{g}
\ar@/^/[rrrd]^-{\gamma} \\
A\times_{B}C \ar@<.5ex>[ru]^-{\pi_2}
\ar@<-.5ex>[rd]_-{\pi_1} && B \ar@<.5ex>[ld]^-{r} \ar@<-.5ex>[lu]_-{s}
 \ar[rr]|-{\beta} && D\\
& A \ar@<.5ex>[ru]^-{f} \ar@<-.5ex>[lu]_-{e_1} \ar@/_/[urrr]_-{\alpha}}}
\end{equation*}
in which the pullback square is a double split epimorphism.

The triple $(\alpha,\beta,\gamma)$ is said to be \defn{admissible with respect to $(f,r,g,s)$} if there is a (necessarily unique) morphism $\varphi\colon{A\times_{B}C\to D}$ such that $\varphi\comp e_{1}=\alpha$ and $\varphi\comp e_{2}=\gamma$.

\subsection*{Commuting pairs in terms of admissibility}
It is immediately clear from the definitions that the morphisms $x$ and $y$ commute over $w$ if and only if the triple $\bigl(\som{w}{x},w,\som{w}{y}\bigr)$ is admissible with respect to $\bigl(\som{1_{W}}{0},\iota_{W},\som{1_{W}}{0},\iota_{W}\bigr)$ as in the diagram
\begin{equation}\label{weight}
\vcenter{\xymatrix@!0@=5em{W+X \ar@<.5ex>[r]^-{\som{1_{W}}{0}} \ar[rd]_-{\som{w}{x}} & W
\ar@<.5ex>[l]^-{\iota_{W}}
\ar@<-.5ex>[r]_-{\iota_{W}}
\ar[d]^-{w} & W+Y \ar@<-.5ex>[l]_-{\som{1_{W}}{0}} \ar[ld]^-{\som{w}{y}}\\
& D.}}
\end{equation}

\subsection*{Admissibility in terms of commuting pairs}
Consider a diagram~\eqref{adm} and the induced weighted cospan
\begin{equation*}\label{Normalisation}
\vcenter{\xymatrix@=3em{& B \ar[d]^-{\beta}\\
X=\Ker(f) \ar[r]_-{\alpha\circ \ker(f)} & D & \Ker(g)=Y. \ar[l]^-{\gamma\circ \ker(g)}}}
\end{equation*}
We claim that the triple $(\alpha,\beta,\gamma)$ is admissible with respect to $(f,r,g,s)$ if and only if $x=\alpha\comp \ker(f)$ and $y=\gamma\comp \ker(g)$ commute over $w=\beta\colon {W=B\to D}$. To see this, it suffices to compare Diagram~\eqref{adm} with the induced Diagram~\eqref{weight}. In fact there is a regular epimorphism of admissibility diagrams from the latter to the former which keeps~$D$ fixed and makes
\[
\vcenter{\xymatrix@!0@=5em{B+X \ar@<.5ex>[r]^-{\som{1_{B}}{0}} \ar@{-{ >>}}[d]_-{\som{r}{\ker(f)}} & B
\ar@<.5ex>[l]^-{\iota_{B}}
\ar@<-.5ex>[r]_-{\iota_{B}}
\ar@{=}[d] & B+Y \ar@<-.5ex>[l]_-{\som{1_{B}}{0}} \ar@{-{ >>}}[d]^-{\som{s}{\ker(g)}}\\
A \ar@<.5ex>[r]^-{f} & B
\ar@<.5ex>[l]^-{r}
\ar@<-.5ex>[r]_-{s} & C \ar@<-.5ex>[l]_-{g}}}
\]
commute. This already proves the ``only if'' in our claim. For the ``if'' suppose that~$x$ and $y$ commute over $\beta$. For the induced arrow
\[
\varphi\colon{(B+X)\times_{B}(B+Y)\to D}
\]
to factor over the regular epimorphism
\[
\som{r}{\ker(f)}\times_{B}\som{s}{\ker(g)}\colon (B+X)\times_{B}(B+Y)\to A\times_{B}C,
\]
we only need that it vanishes on $\Ker\bigl(\som{r}{\ker(f)}\bigr)\times\Ker\bigl(\som{s}{\ker(g)}\bigr)$. This does indeed happen, as
\begin{align*}
\varphi\comp \bigl(\ker\bigl(\som{r}{\ker(f)}\bigr)\times\ker\bigl(\som{s}{\ker(g)}\bigr)\bigr)\comp \langle 1, 0\rangle
&= \varphi\comp \bigl\langle1_{B+X},\iota_{B}\comp \som{1_{B}}{0}\bigr\rangle\comp\ker\bigl(\som{r}{\ker(f)}\bigr)\\
&=\som{\beta}{x}\comp \ker\bigl(\som{r}{\ker(f)}\bigr)\\
&=\alpha\comp \som{r}{\ker(f)}\comp \ker\bigl(\som{r}{\ker(f)}\bigr)
\end{align*}
is trivial. Similarly, one can check that the arrow 
\[
\varphi\comp \bigl(\ker\bigl(\som{r}{\ker(f)}\bigr)\times\ker\bigl(\som{s}{\ker(g)}\bigr)\bigr)\comp \langle 0, 1\rangle
\]
is trivial.

\subsection*{Binary and ternary Higgins commutators}
If $k\colon K \to X$ and $l\colon L \to X$ are subobjects of an object $X$ in a finitely cocomplete homological category, then the \defn{(Higgins) commutator}
 ${[K,L]\leq X}$ is the image of the induced morphism
\[
\xymatrix@=2em{K\tensor L \ar@{{ |>}->}[r]^-{\iota_{K,L}} & K+L \ar[r]^-{\bigsom{k}{l}} & X,}
\]
where
\[
K\tensor L=\Ker\bigl(\left\langle\begin{smallmatrix} 1_{K} & 0 \\
0 & 1_{L}\end{smallmatrix}\right\rangle
\colon K+L\to K\times L\bigr).
\]
As explained in~\cite{GJU}, the Higgins commutator is another special case of the weighted subobject commutator recalled above. This commutator was first introduced in~\cite{Actions, MM-NC}. Higher-order versions of it exist and are studied in~\cite{HVdL, Actions}. 

The object $K\tensor L$, as the ${K\tensor L\tensor M}$ below, is an example of a \defn{co-smash product}~\cite{Smash}. It is worth recalling form~\cite{MM-NC} that it may be computed as the intersection $K\flat L\meet L\flat K$, where the object $K\flat L$ from~\cite{BJK} is the kernel in the split exact sequence
\[
\xymatrix{0 \ar[r] & K\flat L \ar@{{ |>}->}[r] & K+L \ar@<.5ex>@{-{ >>}}[r]^{\som{1_{K}}{0}} & K \ar[r] \ar@<.5ex>[l]^-{\iota_{K}} & 0.}
\]
Furthermore, also the sequence
\begin{equation}\label{Smash as kernel of bemol}
\xymatrix{0 \ar[r] & K\diamond L \ar@{{ |>}->}[r] & K\flat L \ar@<.5ex>@{-{ >>}}[r] & L \ar[r] \ar@<.5ex>[l] & 0}
\end{equation}
is split exact.

If $m\colon M \to X$ is another subobject of $X$, then the \defn{ternary commutator} $[K,L,M]\leq X$ is defined as the image of the composite
\[
\xymatrix@=3em{K\tensor L\tensor M \ar@{{ |>}->}[r]^-{\iota_{K,L,M}} & K+L+M \ar[r]^-{\somm{k}{l}{m}} & X,}
\]
where $\iota_{K,L,M}$ is the kernel of the morphism
\[
\xymatrix@=8em{K+L+M \ar[r]^-{\left\langle\begin{smallmatrix} i_{K} & i_{K} & 0 \\
i_{L} & 0 & i_{L}\\
0 & i_{M} & i_{M}\end{smallmatrix}\right\rangle} & (K+L)\times (K+M) \times (L+M).}
\]
It is well known that co-smash products are not associative, in general; furthermore, ternary co-smash products or commutators need not be decomposable into iterated binary ones: see~\cite{Smash, Actions, HVdL}.

\begin{theorem}\label{Thm GJU}
Consider a weighted cospan~\eqref{weighted cospan} such that $x$ and $y$ are normal monomorphisms (= kernels) in a finitely cocomplete homological category. Then $x$ and $y$ commute over $w$ precisely when the commutators $[X,Y]$ and $[X,Y,\Im(w)]$ vanish.
\end{theorem}
\begin{proof}
First of all we show that $x$ and~$y$ coincide with the images of $\som{w}{x}\comp \ker\bigl(\som{1_{W}}{0}\bigr)$ and $\som{w}{y}\comp \ker\bigl(\som{1_{W}}{0}\bigr)$, respectively, as in~\eqref{weight}. To see this, we consider the diagram with short exact rows
\[
\xymatrix{&& X \ar[d]^-{\iota_{X}} \ar@{.>}[ld]_-{\eta_{X}^{W}} \\
0 \ar[r] & W\flat X \ar@{.>}[d]_-{\xi} \ar@{{ |>}->}[r]_-{\kappa_{B,X}} & W+X \ar@{-{ >>}}[r]^-{\som{1_{W}}{0}} \ar[d]^-{\som{w}{x}} & W \ar[d]^-{d\circ w} \ar[r] & 0\\
0 \ar[r] & X \ar@{{ |>}->}[r]_-{x} & D \ar@{-{ >>}}[r]_-{d} & D_{0} \ar[r] & 0.}
\]
It is clear that $\som{1_{W}}{0}\comp \iota_{X}=0$ induces the factorisation $\eta_{X}^{W}$ of $\iota_{X}$ over the kernel $\kappa_{B,X}$ of $\som{1_{W}}{0}$. Similarly, since
\[
d\comp \som{w}{x}\comp\kappa_{B,X}=d\comp w\comp \som{1_{W}}{0}\comp\kappa_{B,X}
\]
is trivial we obtain the dotted factorisation $\xi$. Now
\[
x\comp \xi\comp \eta_{X}^{W}=\som{w}{x}\comp\kappa_{B,X}\comp \eta_{X}^{W}=\som{w}{x}\comp\iota_{X}=x,
\]
so $\xi\comp \eta_{X}^{W}=1_{X}$ because $x$ is a monomorphism. In particular, $\xi$ is a regular epimorphism. It follows that $x$ is the image of $\som{w}{x}\comp\kappa_{B,X}$.

We know from the above discussion that $x$ and $y$ commute over $w$ precisely when the triple $\bigl(\som{w}{x},w,\som{w}{y}\bigr)$ is admissible with respect to $\bigl(\som{1_{W}}{0},\iota_{W},\som{1_{W}}{0},\iota_{W}\bigr)$. Lemma~4.5 in~\cite{HVdL} now tells us that this happens if and only if the commutators $[X,Y]$ and $[X,Y,\Im(w)]$ vanish.
\end{proof}

Via Theorem~4.6 in~\cite{HVdL} we now recover the known result that the \emph{Smith is Huq} condition~\cite{MFVdL} holds if and only if, for any given cospan of normal monomorphisms $(x,y)$, the property of commuting over $w$ is independent of the chosen weight $w$ making $(x,y,w)$ a weighted cospan.

We also see that the $(W,w)$-weighted normal commutator $N[(X,x),(Y,y)]_{(W,w)}$ of $x$ and $y$ is the normal closure of $[X,Y]\join [X,Y,\Im(w)]$ in $D$, since these two normal subobjects satisfy the same universal property. We shall, however, not insist further on this, because we can obtain the following refinement (Theorem~\ref{Thm Decomposition}).

\begin{lemma}\label{Lemma Smash of Sum}
If $X$, $Y$, and $W$ are objects in a finitely cocomplete homological ca\-te\-gory, then there is a decomposition
\[
(X+Y)\diamond W \cong \bigl((X\diamond Y \diamond W)\rtimes (X\tensor W)\bigr)\rtimes (Y \tensor W).
\]
More precisely, there exists an object $V$ and split short exact sequences
\[
\xymatrix{0 \ar[r] & V \ar@{{ |>}->}[r] & (X+Y)\diamond W \ar@<.5ex>@{-{ >>}}[r] & Y \diamond W \ar[r] \ar@<.5ex>[l] & 0}
\]
and
\[
\xymatrix{0 \ar[r] & X\diamond Y \diamond W \ar@{{ |>}->}[r] & V \ar@<.5ex>@{-{ >>}}[r] & X \diamond W \ar[r] \ar@<.5ex>[l] & 0.}
\]
\end{lemma}
\begin{proof}
This is Lemma~2.12 in~\cite{HVdL}, a result which was first obtained by M.~Hartl and B.~Loiseau.
\end{proof}

\begin{theorem}\label{Thm Decomposition}
Given a weighted cospan~\eqref{weighted cospan} in a finitely cocomplete homological category, the $(W,w)$-weighted subobject commutator of $x$ and $y$ decomposes as
\[
[(X,x),(Y,y)]_{(W,w)}= [X,Y]\join [X,Y,\Im(w)].
\]
\end{theorem}
\begin{proof}
We decompose the kernel $K$ of the short exact sequence~\eqref{Kernel K} into a join of the co-smash products $X\tensor Y$ and $X\tensor Y\tensor W$ considered as subobjects of $K$. The result then follows from the compatibility of the ternary co-smash product with image factorisations (Corollary 2.14 in~\cite{Actions} and the fact that co-smash products preserve monomorphisms). Indeed, the image of the composite ${X\tensor Y\tensor W\to W+X+Y\to D}$ is $[X,Y,\Im(w)]$, which is a subobject of $[(X,x),(Y,y)]_{(W,w)}$. It is easily seen that also $[X,Y]\leq [(X,x),(Y,y)]_{(W,w)}$ and that these two inclusions are jointly regular epic. 

Consider the cube of solid split epimorphisms
\[
\xymatrix@!0@C=3.5em@R=3.5em{&& K \ar@{.>}[ddd] \ar@{{ |>}.>}[rrr] \ar@{{ |>}.>}[dl] &&& (W+X)\flat Y \ar@{.>}[ddd] \ar@{{ |>}.>}[dl] \ar@{.>}[rrr] &&& W\flat Y \ar@{.>}[ddd] \ar@{{ |>}.>}[dl]\\
& (W+Y)\flat X \ar@{.>}[ddd] \ar@{.>}[dl] \ar@{{ |>}.>}[rrr] &&& W+X+Y \ar[rrr] \ar@{->}[ddd]|(.33){\hole} \ar[ld] &&& W+Y \ar[ddd] \ar[ld] \\
W\flat X \ar@{.>}[ddd] \ar@{{ |>}.>}[rrr] &&& W+X \ar[rrr] \ar[ddd] &&& W \ar[ddd]\\
&& X\tensor Y \ar@{{ |>}.>}[rrr] \ar@{{ |>}.>}[dl] &&& X\flat Y \ar@{{ |>}.>}[ld] \ar@{.>}[rrr] &&& Y \ar@{:}[dl] \\
& Y\flat X \ar@{.>}[dl] \ar@{{ |>}.>}[rrr] &&& X+Y \ar@{->}[rrr]|(.66){\hole} \ar@{->}[ld] &&& Y \ar[ld]\\
X \ar@{:}[rrr] &&& X \ar[rrr]  &&& 0}
\]
which, taking kernels horizontally, yields two $3\times 3$ diagrams (or, equivalently, a~$3\times 3$ diagram of vertical split epimorphisms). Note that the bottom one has $X\tensor Y$, and the top one $K$, in its back left corner. It suffices to prove that, taking kernels vertically now, we obtain the split exact sequence
\[
\xymatrix{0 \ar[r] & X\tensor Y\tensor W \ar@{{ |>}->}[r] & K \ar@<.5ex>@{-{ >>}}[r] & X\tensor Y \ar[r] \ar@<.5ex>[l] & 0}
\]
in the back left corner of the induced $3\times 3\times 3$ diagram. Taking vertical kernels of the front and middle sections of the diagram above, we already obtain a morphism
\[
\xymatrix{0 \ar[r] & U \ar[d] \ar@{{ |>}->}[r] & (X+Y)\flat W \ar@<.5ex>@{-{ >>}}[r] \ar[d] & Y\flat W \ar[d] \ar[r] \ar@<.5ex>[l] & 0 \\
0 \ar[r] & X\diamond W \ar@{{ |>}->}[r] & X\flat W \ar@<.5ex>@{-{ >>}}[r] & W \ar[r] \ar@<.5ex>[l] & 0}
\]
of short exact sequences. Using~\eqref{Smash as kernel of bemol} we see that the sequence
\[
\xymatrix{0 \ar[r] & U \ar@{{ |>}->}[r] & (X+Y)\diamond W \ar@<.5ex>@{-{ >>}}[r] & Y \diamond W \ar[r] \ar@<.5ex>[l] & 0}
\]
is split exact. Noting that $V$ in Lemma~\ref{Lemma Smash of Sum} is the object $U$, we see that the co-smash product $X\diamond Y \diamond W$ must coincide with the kernel of $U\to X\diamond W$, which we already know coincides with the needed kernel of $K\to X\diamond Y$.
\end{proof}

\subsection*{Acknowledgement}
We would like to thank the referee who guided us through a major revision of the article and thus improved the presentation considerably.



\providecommand{\noopsort}[1]{}
\providecommand{\bysame}{\leavevmode\hbox to3em{\hrulefill}\thinspace}
\providecommand{\MR}{\relax\ifhmode\unskip\space\fi MR }
\providecommand{\MRhref}[2]{%
  \href{http://www.ams.org/mathscinet-getitem?mr=#1}{#2}
}
\providecommand{\href}[2]{#2}

\end{document}